\documentclass[12pt]{amsart}

\usepackage{amsthm}
\usepackage{amssymb}
\usepackage{amsmath}
\numberwithin{equation}{section}
\usepackage{graphics,graphicx,colortbl}
\definecolor{Mygrey}{gray}{0.8}
\newcommand{\bea}{\begin{eqnarray}}
\newcommand{\eea}{\end{eqnarray}}
\newcommand{\be}{\begin{eqnarray*}}
\newcommand{\ee}{\end{eqnarray*}}
\newtheorem{theorem}{Theorem}[section]
\newtheorem{lemma}{Lemma}[section]
\newtheorem{corollary}{Corollary}[section]

\newtheorem{proposition}{Proposition}[section]
\newtheorem{example}{Example}[section]

\begin{document}
\title[Dynamics of Linear Systems over Finite Commutative Rings]{ Dynamics of Linear Systems \\ over Finite Commutative Rings}
\author[Yangjiang Wei]{Yangjiang Wei}
\address{School of Mathematical Sciences, Guangxi Teachers Education University, Nanning 530023, P. R. China}
\email{gus02@163.com}
\author[Guangwu Xu]{Guangwu Xu}
\address{Department of Electrical Engineering and Computer Sciences, University of Wisconsin, Milwaukee, WI 53201, USA} 
\email{gxu4uwm@uwm.edu}
\author[Yi Ming Zou]{Yi Ming Zou$^*$}
\address{Department of Mathematical Sciences, University of Wisconsin, Milwaukee, WI 53201, USA} \email{ymzou@uwm.edu}
\thanks{* Corresponding author. Email: ymzou@uwm.edu}
\keywords{Finite rings, Modules, Linear systems, Cycle lengths}
\subjclass[2010]{13M99, 13P99, 15B33}
\maketitle
\begin{abstract}
The dynamics of a linear dynamical system over a finite field can be described by using the elementary divisors of the corresponding matrix. It is natural to extend the investigation to a general finite commutative ring. In a previous publication, the last two authors developed an efficient algorithm to determine whether a linear dynamical system over a finite commutative ring is a fixed point system or not. The algorithm can also be used to reduce the problem of finding the cycles of such a system to the case where the system is given by an automorphism. Here, we further analyze the cycle structure of such a system and develop a method to determine its cycles.
 
\end{abstract}
\section{Introduction}
\par
Let $R$ be a commutative ring with $1$, let $M$ be an $R$-module, and let $f: M\longrightarrow M$ be an $R$-module endomorphism. We may consider $f$ as a dynamical system via iteration:
\be
f,\;\; f^2 = f\circ f,\;\; \cdots, \;\;f^m=\underbrace{f\circ f\circ\cdots\circ f}_{\mbox{$m$ copies}},\;\; \cdots,
\ee
and investigate the behaviors of the system. For our convenience, we set $f^0 = id$. By studying the dynamics of the system, we mean to investigate its possible cycles (including fixed points) generated by the elements of $M$. It is clear that if the cardinality of $M$ is infinite, then there may not be any cycle except the obvious fixed point $0$. But if $M$ is finite, then for any initial input $a\in M$, the sequence $f^m(a)$, $m\ge 0$, will eventually either stabilize or enter a cycle. We should restrict our attention to the case where both $R$ and $M$ are finite in this paper, and we call such a system {\it a linear dynamical system over the finite ring $R$}. 
\par\medskip
If $R = F$ is a finite field, then $M$ is a finite dimensional vector space over $F$, and the dynamics of $f$ can be described by using the elementary divisors of the $F$-linear map $f$. Suppose the elementary divisors of $f$ are (see \cite{Hu}, pp. 356-357):
\be
p_i^{m_{ij}},\;\; 1\le i\le s,\;\; 1\le j\le k_i,
\ee
where each $p_i$ is an irreducible polynomial in $F[x]$, and the $m_{ij}$ are integers such that for each $i$,
\be
m_{i1}\ge m_{i2}\ge\cdots\ge m_{ik_i}> 0.
\ee
Then there are $f$-cyclic subspaces $M_{ij}$ (i.e. there exists a vector $v\in M_{ij}$ and an integer $k\ge 0$ such that $(v, f(v), \ldots, f^k(v))$ is a basis of $M_{ij}$) with minimal polynomial $p_i^{m_{ij}}$, $1\le i\le s,\; 1\le j\le k_i$, such that $M$ is a direct sum of the $M_{ij}$'s. 
\par 
It is easy to see (over any ring $R$) that if $M = M_1\oplus M_2$ is a direct sum decomposition of $f$-invariant submodules (we will just call them $f$-submodules), and $C_i$ is a cycle in $M_i$ of length $c_i$ ($i=1,2$), then the direct product
\be
C_1\times C_2 := \{(a,b)\;|\; a\in M_1,\;b\in M_2\}
\ee
offers $GCD(c_1,c_2)$ cycles of length $LCM(c_1,c_2)$ for the module $M$. So we can just work with each direct summand independently.   
\par
Therefore, the problem of finding the cycles (including fixed points) of $f$ (over a field $F$) reduces to the problem of finding the cycles of the restricted linear maps 
\be
f_{ij} = f|_{M_{ij}}\; : \;M_{ij}\longrightarrow M_{ij},\;\; 1\le i\le s,\;\; 1\le j\le k_i.
\ee
\par
If $p_i = x$, the corresponding linear maps $f_{ij},\; 1\le j\le k_i$, are nilpotent and thus all have just one fixed point $0$ and have no cycle of length $> 1$. So the question is to find the cycle structures of those $f_{ij}$ such that the corresponding $p_i$ are not $x$ (these $f_{ij}$ are automorphisms of the $M_{ij}$). For these cases, since the characteristic polynomial of each $f_{ij}$ is equal to its minimal polynomial, Elspas' formula describes its cycle structure \cite{Els}\cite{He} (we use $g_{ij}^k$ instead of $p_i^{m_{ij}}$ to simplify the notation in the formula):
\be
G_{ij} = 1 + \sum_{k=1}^{s}\frac{q^{mk}-q^{m(k-1)}}{c_k}\mathcal{C}_k,
\ee
where $q = |F|$, $1$ corresponds to the fixed point $0$, $\mathcal{C}_k$ is a cycle of length $c_k$, and $c_k = \mbox{ord}(g_{ij}^k)$ (i.e. $c_k$ is the least positive integer $t$ such that $g_{ij}^k$ divides $x^t - 1$).  This formula gives a nice description of the cycle structure, though the actual computation is more involved: one first computes the normal form of the matrix $xI - A$, where $A$ is the matrix of $f$ with respect to a basis of $M$, or computes its minimum polynomial $m(f)$, then factors the polynomials that appear in the normal form of $xI - A$ (or $m(f)$) to get the $g_{ij}^k$, and then computes the order of each $g_{ij}^k$.  
\par
Note that if $g_{ij}^k$ divides $x^t-1$, then all $g_{ij}^s,\;0\le s\le k$, divide $x^t-1$. So by the comment about the cycles of a direct sum of $f$-modules, we have:
\begin{lemma}\label{L11}
The $F$-module endomorphism $f$ possesses a cycle of maximum length, say $c$, such that all other cycle lengths are divisors of $c$.
\end{lemma}
\par
For a general finite commutative ring $R$, since the factorization in $R[x]$ may not be unique, different approaches are needed. The study of linear dynamical systems over rings of the form $\mathbb{Z}_q$ was suggested in \cite{Co2} due to its relation with monomial dynamical systems over finite fields. In \cite{Bo} an approach by using Fitting's lemma \cite{Hu}: there is a positive integer $k$ such that $M = \mbox{Ker}f^k\oplus\mbox{Im}f^k$  was suggested. However, the key to the success of the proposed approach, which is to find a proper $k$ and compute $f^k$ efficiently, was not treated. In \cite{XZ}, an upper bound for $k$ was determined, and an algorithm that runs in time $O(n^3\log(n\log(q)))$, where $n$ is given by $M = R^n$ and $q =|R|$, to compute $f^k$ was developed. The results of \cite{XZ} provide a solution to the problem of determining whether or not such a system is a fixed point system, and also reduce the problem of finding the cycle structure of such a system to that of an automorphism (since $f:\mbox{Im}f^k\longrightarrow \mbox{Im}f^k$ is bijective). 
\par
The main idea of \cite{XZ} is, to determine the cycles of $f$, we can just consider the $f$-module $M$ as an abelian group ($\mathbb{Z}$-module) and determine the lengths of the cycles. After obtaining the information on the lengths of the cycles, we can find the cycles by solving linear systems of the form $(f^c - I)X = 0$, where $I$ is the identity map and $c$ is a cycle length. By the structure theorem of finitely generated abelian groups \cite{Hu}, we can assume that
\bea\label{E1}
\mbox{Im}f^k =\mathbb{Z}_{p_1^{a_1}}\oplus \mathbb{Z}_{p_2^{a_2}}\oplus\cdots\oplus\mathbb{Z}_{p_m^{a_m}},
\eea
where $p_1,\ldots, p_m$ are (not necessarily distinct) primes, and $a_1,\ldots, a_m$ are (not necessarily distinct) positive integers. 
\par
We further note that an $f$-module as in (\ref{E1}) can be written as a direct sum of $f$-submodules such that each of these submodules is formed by grouping those $\mathbb{Z}_{p_i^{a_i}}$'s with the same prime $p_i$ together. As discussed before, whenever we have a direct sum of $f$-submodules, we can just work with each direct summand independently. Therefore, our goal here is to analyze the cycle structure of an automorphism of an abelian group of the form ($p$ is a prime):
\bea\label{E2}
\hspace{0.6cm} M = \mathbb{Z}_{p^{a_1}}\times \mathbb{Z}_{p^{a_2}}\times\cdots\times\mathbb{Z}_{p^{a_m}},\; 1\le a_1\le a_2 \le \ldots\le a_m.
\eea
Since abelian groups are $\mathbb{Z}$-modules, an endomorphism of $M$ can be represented by an integer matrix $A$; and if all $a_i = 1$, i.e. $M = \mathbb{Z}^m_p$, it is an endomorphism of the vector space $M$ over the field $\mathbb{Z}_p$.  
In the case where all $a_i$'s are the same, i.e. $M=\mathbb{Z}^m_{p^a}$, an approach was developed in \cite{De} for an arbitrary endomorphism of $M$ using number theory techniques based on congruence of integers.
\par
In this paper, we consider the general case where $M$ is defined by  (\ref{E2}) such that the $a_i$'s are not necessary the same. Our emphasis is on the computation of the cycles; and for that, we need to be able to compute the possible cycle lengths efficiently. In Section 2, we give an upper bound for the cycle lengths such that all possible cycle lengths are factors of the given upper bound. The factors of the given upper bound are obtained from the induced linear systems over the finite field $\mathbb{Z}_p$, and thus there is no extra work needed to factor the upper bound. The results of this paper and the results in \cite{XZ} together provide an algorithm to determine the cycles structure of a linear dynamical system over a finite commutative ring. The algorithm is given in Section 3. We consider three examples in Section 4. The first example shows that the given upper bound for the cycle lengths is sharp, and the third example is a linear system derived from a real world cancer regulatory network. 
\section{Result}
We begin by collecting some general facts. The following lemma holds for any ring $R$.
\par
\begin{lemma}\label{L21}
 Let $M, N$ be $R$-modules, let $\varphi: M \rightarrow N$ be an onto $R$-module homomorphism, and let $K = \ker\varphi$. If $f: M \rightarrow M$ is an $R$-module endomorphism such that $f(K)\subset K$, then $f$ induces an $R$-module endomorphism $\bar{f}$ of $N$ by $\bar{f}(\varphi(m)) = \varphi(f(m)),\; m\in M$. 
\end{lemma}
\begin{proof} 
The proof is a straightforward verification of the definition. 
\end{proof}
 \par\medskip
  Recall that we work with a finite commutative ring $R$ and a finite $R$-module, so it is possible to use a reduction approach to find the cycles of an $R$-module endomorphism $f$. For this purpose, we investigate the relationships among the cycles of the $f$-modules $M$, $N$, and $K$ (notation as in Lemma 2). For any $m\in M$, we denote by $\ell(m)$ the length of the cycle generated by $m$ (if $m$ lies on a cycle). We have the following simple fact about fixed points:
 \par
 \begin{proposition} Notation as in Lemma \ref{L21}. Let $v$ be a fixed point of $N$ and let $u$ be a fixed point of $M$ such that $\varphi(u) = v$. Then $u_1$ is a fixed point of $M$ such that $\varphi(u_1) = v$ if and only if $u_1 = u+w$ for a fixed point $w\in K$.  In particular, the number of fixed points $u$ of $M$ such that $\varphi(u) = v$ is $0$ or equal to the number of fixed points of $K$. 
 \end{proposition}
 \begin{proof} 
 We have that $u_1 = u+w$ is a fixed point $\Leftrightarrow$ $f(u+w) = u+f(w) =u+w$ $\Leftrightarrow$ $f(w)=w$. 
 \end{proof}
 \par\medskip
From now on, we assume that $f$ is an automorphism. Then every element of $M$ lies in a cycle. Note that by our assumption on $M$, $p^{a_m}(M) = 0$ (see (\ref{E2})), thus $p^{a_m}(K) = 0$. Suppose that $a$ is the least positive integer $i$ such that $p^i(K) = 0$.  For $u\in M$, let $\varphi(u) = v\in N$, and let $\ell(v)=s$. With these assumptions, we have $f^{s}(u) = u+w$ for some $w\in K$. Let $\ell(w) = k$.  We will write $[s,k]$ (respectively, $(s,k)$) for the least common multiple (respectively, the greatest common divisor) of $s$ and $k$.
 \begin{lemma}\label{L22}
 There exists $0\le b\le a$ such that $\ell(u)=p^b[s,k]$.
 \end{lemma}
 \begin{proof} 
 Let $\ell(u) = c$. Since $f^c(u) = u$, $\bar{f}^c(v) = v$ in $N$. So $c$ is a multiple of $s$. Let $c=ts$, then from $ f^{s}(u) = u+w$ we have $w = f^s(u) - u$ and
 \be
 u = f^{ts}(u) = f^{(t-1)s}(u) + f^{(t-1)s}(w).
 \ee
 Applying $f^s$ to this equation we have $f^{c}(w) = w$, which implies that $c$ is also a multiple of $k$. Thus $c$ is a common multiple of $s$ and $k$, say $c=c_0[s,k]$ for some positive integer $c_0$. This holds without the assumption that $p^a(K)=0$. If $p^a(K)=0$, we will have $c_0 =p^b$ for some $0\le b\le a$. 
 \par
 To see that, let $[s,k]=ns$, where $n=\frac{k}{(s,k)}$. Start with $ f^{s}(u) = u+w$ and compute inductively, we have 
 \bea\label{E3}
 f^{ns}(u) = u+w+f^s(w)+\cdots+f^{(n-1)s}(w). 
 \eea
 Note that since $ns$ is a multiple of $k$, $(f^{ns}-I)(w)=0$, which implies 
 \be
 (f^s - I)(w+f^s(w)+\cdots+f^{(n-1)s}(w)) = 0,
 \ee
 that is\be
 w_1:=w+f^s(w)+\cdots+f^{(n-1)s}(w)
 \ee
 is a fixed point of $f^s$. If $w_1=0$, then $c=[s,k]$. However, we do not know if $w_1=0$ or not, so we use the assumption that $p^a(K)=0$. Start from (\ref{E3}): $f^{ns}(u) = u+w_1$, note that $w_1$ is a fixed point of $f^{ns}$, we have
 \be
 f^{p^ans}(u) = (f^{ns})^{p^a}(u) = u+p^aw_1 = u. \quad \mbox{(since $w_1\in K$)}
 \ee
 Thus $\ell(u)$ is a factor of $p^ans$, which implies $c_0=p^b$ for some $0\le b\le a$. 
\end{proof}
  \par\medskip
We now assume that $f:M\longrightarrow M$ is a $\mathbb{Z}$-module automorphism, where $M$ is defined by (\ref{E2}). Note that $p^{a_m}(M)=0$. For each $0\le i\le a_m$, we define
\be
M_i =\{v\in M\;|\; p^iv=0\}.
\ee
Then each $M_i$ is an $f$-submodule of $M$. Consider the $f$-submodule filtration
\bea\label{E4}
M_0=(0)\subset M_1\subset\cdots\subset M_{a_m-1}\subset M_{a_m}=M.
\eea
There are positive integers $k_i \le m,\; 1\le i\le a_m$, such that the quotients
\bea\label{E5}
M_i/M_{i-1}\cong\mathbb{Z}_p^{k_i}.
\eea
The automorphism $f$ induces automorphisms
\be
\bar{f}_i\; :\; M_i/M_{i-1}\cong\mathbb{Z}_p^{k_i}\longrightarrow
M_i/M_{i-1}\cong\mathbb{Z}_p^{k_i},\;\; 1\le i\le a_m.
\ee
Since $\mathbb{Z}_p$ is a field, the dynamics of these automorphisms can be determined as discussed in the introduction section. For each $1\le i\le a_m$, let the set of cycle lengths of $\bar{f}_i$ be $L_{ij} =\{c_{ij}\;:\; 1\le j\le \ell_i$\},
where $\ell_i$ is a positive integer depending on $i$. We assume that $c_{i1}> c_{ij},\; 1\le i\le a_m,\; 1<j\le \ell_i$. Then by Lemma \ref{L11}, for any $1\le i\le a_m$, $c_{ij}|c_{i1}$ for all $1\le j\le \ell_i$.
\par
\begin{theorem}\label{T1} 
Assume that $f:M\longrightarrow M$ is a $\mathbb{Z}$-module automorphism, where $M$ is defined by (\ref{E2}), and keep the notation introduced above. 
\par
(1) For each integer $c_{i1}$ ($1\le i\le a_m$), the dynamical system $f:M\longrightarrow M$ possesses a cycle whose length is a multiple of $c_{i1}$.
\par
(2) All cycle lengths of $f$ are factors of 
\bea\label{c1} 
p^{a_m-1}\cdot\mbox{LCM}(c_{11}, c_{21},\ldots, c_{a_{m}1}), 
\eea
where LCM stands for the least common multiple. 
\end{theorem}
\begin{proof} 
(1) For a fixed $1\le i\le a_m$, choose $u\in M_i$ such that its image $\bar{u}\in M_i/M_{i-1}$ belongs to a cycle of length $c_{i1}$. If the cycle to which $u$ belongs has length $c$, then $f^c(u)=u$ and hence $f^c(\bar{u})=\overline{f^c(u)}=\bar{u}$  in $M_i/M_{i-1}$, which implies that $c_{i1}|c$.
\par
(2) We use induction on $a_m$. There is nothing to prove if $a_m=1$ since $f$ is just an automorphism of the vector space $M=\mathbb{Z}^m_{p}$. Let $a_m \ge 2$. Then since $f$ induces an automorphism of each $f$-submodule $M_i$ in (\ref{E4}), it induces an automorphism of 
\be
M/M_1 \cong \mathbb{Z}_{p^{a_1-1}}\times \mathbb{Z}_{p^{a_2-1}}\times\cdots\times\mathbb{Z}_{p^{a_m-1}}. 
\ee
The filtration of (\ref{E4}) induces the following filtration:
\bea
\hspace{0.6cm} (0)=M_1/M_1\subset M_2/M_1\subset\cdots\subset M_{a_m-1}/M_1\subset M/M_1.
\eea
Since for $2\le i\le a_m$
\bea
(M_i/M_1)/(M_{i-1}/M_1)\cong M_i/M_{i-1},
\eea
by the induction assumption, we have that the possible cycle lengths of $M/M_1$ are factors of $r=p^{a_m-2}LCM(c_{21},\ldots, c_{a_m1})$. Now the theorem follows by noticing $p(M_1)=0$ and applying Lemma \ref{L11} to $M_1$ and Lemma \ref{L21} to the projection $M\rightarrow M/M_{1}$.  
\end{proof}
\par\medskip
\begin{corollary}\label{C21} 
If $M = \mathbb{Z}_{p^a}^m$ (i.e. the case where all $a_i$ are equal), the cycle lengths are factors of $p^{a-1}c$, where $c$ is the maximum cycle length of $M/M_{a-1}\cong \mathbb{Z}_p^m$ with the module structure induced by $f$.
\end{corollary} 
\begin{proof}
Consider the onto $f$-module homomorphism $\varphi : M\rightarrow M_i/M_{i-1}$ defined by first taking $m\rightarrow p^{a-i}m,\; m\in M$, and then follows by the projection from $M_i$ onto $M_i/M_{i-1}$. This homomorphism induces an $f$-module  isomorphism $M/M_{a-1}\rightarrow M_i/M_{i-1}$ since its kernel is $M_{a-1}$. Thus all $M_i/M_{i-1}$ $(2\le i\le a$) are isomorphic as $f$-modules and have the same cycle structure, which implies that $c_{i1} = c$ for all $i$, and the corollary follows from (2) of the theorem.  
\end{proof}
\par\medskip
We remark that if $f$ is not an automorphism then the above corollary does not apply, since not every element of $M$ lies in a cycle.
\section{Computing the Cycles}
\par
Under the assumption that the cycle lengths of a linear system over a finite field can be computed, we can use Theorem \ref{T1} to find the cycles of a linear system $f : M\rightarrow M$ over a finite commutative ring $R$. Our algorithm is divided into four steps.
\par\medskip
\begin{enumerate}
\item {\bf Reduce to an automorphism.} By Theorem 2.1 of \cite{XZ}\footnote{The assumption on the module considered in Theorem 2.1 \cite{XZ} is somewhat different, but a quick check of the proof there reveals that it also applies to the case here.}  $N=m\log_2(q)$ satisfies $f^{N+1}(M) = f^N(M)$, where $q = |R|$, and $m$ is given by a presentation of the $R$-module $M$ as a quotient of the free module $R^m$:
\be
\varphi : R^m \longrightarrow M\;\;\mbox{and}\;\; R^m/\ker(\varphi) \cong M.
\ee
Therefore we have an induced automorphism $f: f^N(M)\rightarrow f^N(M)$. Consider the structure of the finitely generated abelian group $f^N(M)$, and reduce it to the case where the module is defined by (\ref{E2}). Thus we assume the module $M$ is defined by (\ref{E2}) in steps 2 -- 4.
\par\medskip
\item
{\bf Compute the number $c_{i1}$ for each induced automorphism (see the paragraph precedes Theorem 1)} 
\bea\label{E7}
\bar{f}_i : \mathbb{Z}_{p}^{k_i}\rightarrow \mathbb{Z}_{p}^{k_i}, 1\le i\le a_m.
\eea 
If the minimal polynomial of $\bar{f}_i$ is $g_i$ and $g_i=g_{i1}^{r_{i1}}\cdots g_{is_i}^{r_{is_i}}$ is its irreducible factorization, then $c_{i1} = LCM(h_{i1},\ldots, h_{is_i})$, where $h_{ij} = ord(g_{ij}^{r_{ij}})$. Since $\mathbb{Z}_p$ is a field, one can use the existing approaches \cite{Els}\cite{He}. See also the discussions later.
\par\medskip
\item {\bf Compute the order of $f$ as an element of the automorphism group of $M$.} Let $n :=p^{a_m-1}\cdot\mbox{LCM}(c_{11}, c_{21},\ldots, c_{a_{m}1})$. Then $n$ is a multiple of the order of $f$. Since the factorization of $n$ is known from step 2 and step 3, we can use the following procedure to compute the order of $f$:
\begin{center}
\begin{tabular}{|lp{8cm}|}
\hline
Input: & $n$ --- a multiple of the order of $f$.\\
Output: & $ord$ --- the order of $f$.\\
\hline
& $ord \gets n$;\\
& {\tt while ( 1 ) do} \\
& \hspace{10mm} $ordersmaller \gets 0$;\\
& \hspace{10mm} {\tt  for ( each prime factor $p$ of $ord$ )}\\
& \hspace{20mm} {\tt  if ( $f^{\frac{ord}p} = I$ )}\\
& \hspace{30mm}  $ord \gets \frac{ord}p$; \\
& \hspace{30mm}  $ordersmaller \gets 1$; \\
& \hspace{30mm}  {\tt break;} \\
& \hspace{10mm} {\tt  if ( $ordersmaller = 0$ )}\\
&\hspace{20mm} {\tt  return $ord$;}\\
\hline
\end{tabular}
\end{center}
\par\medskip 
\item {\bf Compute the cycle lengths and the cycles.} Let the order of $f$ be $\mathbf{a}$. Then the cycle lengths of $f$ are given by the factors of $\mathbf{a}$, which are known from step 2. For a factor $d$ of $\mathbf{a}$, solve  the linear system $(f^d - I)X=0$ to obtain the cycles of length $d$. 
\end{enumerate}
\par\medskip
{\bf Analysis of the algorithm.} We ignore the factors contributed by $q = |R|$ in our analysis, since for application purposes we can assume that $q$ is relatively small. 
\par
In step 1, the computation of $f^N$ takes time $O(m^3)$. Also, to find the structure of the abelian group $f^N(M)$, it requires to compute the Smith normal form of the integer matrix of $f: f^N(M)\rightarrow f^N(M)$. It is well-known that the computation of the Smith normal form takes time $O(m^3)$.
\par
In step 3, the given algorithm computes the order of $f$ in time $O(m^3)$ since computing $f^{\mathbf{a}/p}$ requires at most time $ O(m^2\sqrt{m}\log (m))$.
\par
The computations in step 4 are straightforward.
\par
Since we could not find any readily accessible reference on the algorithm and its analysis for the case where $R =\mathbb{Z}_p$ is a field, we give a discussion of the complexity of step 2 as follows.
\par
(1) There is a randomized algorithm \cite{NP} to compute the minimal polynomial $m(x)$ of a liner map
$T: \mathbb{Z}_{p}^k\rightarrow \mathbb{Z}_{p}^k$ in time $O(k^3)$ and uses $O(k)$ random vectors. Since $k\le m$, this computation requires expected time no worse than $O(m^3)$.
\par
(2) To factorize $m(x)$,  some algorithms from \cite{GG} can be used. Assume the degree of $m(x)$ is $d$
and let ${\bf M}(d)$ be the number of field operations needed for multiplying two polynomials
of degree at most $d$,  then $m(x)$ can be factorized in expected number of $O(d{\bf M}(d)\log (dp))$ field
operations. Since $d\le m$, this requires time no worse than $O(m^3)$.
\par
 (3) Note that if $g(x)$ is an irreducible polynomial with $\mbox{ord}(g) = \alpha$, then
$\mbox{ord}(g^u) = \alpha p^c$ where $c$ is the smallest integer such that $p^c\ge u$ \cite{He}.
\par
Thus, our algorithm runs in expected time $\sim \tilde{O}(m^3)$.
\par\medskip
\section{Conclusion and Examples}
We have developed an efficient algorithm to determine the cycle structure of a linear dynamical system over a finite commutative ring with identity and thus provided a solution to this problem. 
\par
We conclude our paper by considering three examples. The first example shows that the bound given by (\ref{c1}) is sharp.
\begin{example}
Let $R = \mathbb{Z}_{16}$, let $M=\mathbb{Z}_{4}\times\mathbb{Z}_{8}\times\mathbb{Z}_{16}$, and let $f : M\longrightarrow M$ be defined by
\be
A=\begin{pmatrix}
  1 & 1 & 0 \\
  0 & 1 & 1 \\
  0 & 0 & 1 \\
\end{pmatrix}
\ee
with respect to the generators $e_1 = (1,0,0)^t,\; e_2 = (0,1,0)^t,\; e_3=(0,0,1)^t$ of $M$. Then $a_m = 4$,
\begin{gather*}
M_0 = (0) \subset M_1 = <2e_1, 4e_2, 8e_3> \subset M_2 = <e_1, 2e_2, 4e_3>\\
\subset M_3 = <e_1,e_2,2e_3> \subset M_4 =M,
\end{gather*} 
and
\begin{gather*}
M_4/M_3 \cong \mathbb{Z}_2,\; M_3/M_2 \cong \mathbb{Z}_2\times\mathbb{Z}_2,\\
 M_2/M_1 \cong M_1/M_0 \cong \mathbb{Z}_2\times\mathbb{Z}_2\times\mathbb{Z}_2.
\end{gather*}
It is easy to see that the restriction of $f$ to any of these quotients is the identity map, so $c_{i1} = 1$ for $1\le i\le 4$. Thus according to Theorem \ref{T1}, the possible cycle lengths are factors of $2^3$. For this example, there are cycles for each of the possible factors of $2^3$. For instance, the cycle generated by $2e_1+2e_2$ has length $2$, the cycle generated by $e_2$ has length $4$, and the cycle generated by $e_3$ has length $8$. 
\end{example}
\par\medskip
The second example shows that there may not be a cycle of length $p^bc$ (see Corollary \ref{C21} for notation) for any $b>0$.
\begin{example} For any prime $p$ and any integers $a,n > 1$, consider the $\mathbb{Z}_{p^a}$-module $\mathbb{Z}_{p^a}^n$ and the automorphism $f$ defined by the cyclic permutation of the standard basis $(e_1,e_2,\ldots, e_n)$:
\be
e_1 \rightarrow e_{2}\rightarrow e_{3}\rightarrow \cdots\rightarrow e_{n} \rightarrow e_{1}.
\ee
Then the order of $f$ is $n$ and the cycle lengths are the factors of $n$ (for each $d|n$, the non-zero element $e_1+e_{d+1}+e_{2d+1} + \cdots + e_{n-d+1}$ generates a cycle of length $d$).
\end{example}
\par\medskip
Next, we use a simplified version of the network given by the diagram of Fig. 2B in \cite{Zhang} to construct an example. The original network consists of $29$ nodes, the $7$ red output nodes that having a single input node are omitted to make the presentation more streamline (these output nodes add very little extra to the computation). The omitted nodes are: FLIP, A20, RANTES, FasT, GZMB, MEK, and LCK. Our purpose here is to give an example from a real world network, finding a good linear model over a finite ring for the underlying biological system is beyond our scope here. We introduce the variables as in Table 1.
\par\medskip
\begin{table}[h]
  \begin{center}
  {\renewcommand{\arraystretch}{1.4}\small{\tiny 
  \begin{tabular}{|c|c|c|c|c|c|c|c|c|}
    \hline
   IL15	&  RAS	& ERK	&    JAK	& IL2RBT	& STAT3	& IFNGT	& FasL\\ \hline 	
   $x_1$ & $x_2$ & $x_3$ & $x_4$ & $x_5$	& $x_6$	& $x_7$	& $x_8$\\ \hline\hline

   PDGF  &  PDGFR &   PI3K	 &   IL2	&   BcIxL	  &       TPL2	  &  SPHK &	S1P	  \\ \hline	
   $x_9$ & $x_{10}$ & $x_{11}$ & $x_{12}$ & $x_{13}$ & $x_{14}$ & $x_{15}$ & $x_{16}$ \\ \hline\hline

   sFas   & Fas &  DISC  &  Caspase  &  Apoptosis	& IL2RAT &\cellcolor{Mygrey} &\cellcolor{Mygrey}\\ \hline
   $x_{17}$ & $x_{18}$ & $x_{19}$ & $x_{20}$ & $x_{21}$ & $x_{22}$ &\cellcolor{Mygrey} &\cellcolor{Mygrey}\\ \hline 
 \end{tabular}}
  }
  \label{tab: pairing2}
  \caption[Legend of variable names.]{\footnotesize Legend of variable names of the network given by Fig. 2B in \cite{Zhang} (with 7 red output nodes omitted).}
    \end{center}
\end{table} 
\par
\begin{example}
The base ring is $\mathbb{Z}_9$ and the state space is $\mathbb{Z}_9^{22}$. The update functions of the nodes are provided in Table 2.
\par
\begin{table}[h]
  \begin{center}
  {\renewcommand{\arraystretch}{1.4}\small{\tiny 
  \begin{tabular}{|c|c|c|c|c|c|c|c|c|}
    \hline
  $f_1$ & $f_2$ & $f_3$ & $f_4$ & $f_5$	& $f_6$	& $f_7$	& $f_8$\\ \hline 	
   $x_1$ & $x_1$ & $x_2$ & $x_1$ & $x_1$	& $x_4$	& $3x_5+x_6$	& $2x_3+x_5+3x_{14}$\\ \hline\hline

   $f_9$ & $f_{10}$ & $f_{11}$ & $f_{12}$ & $f_{13}$ & $f_{14}$ & $f_{15}$ & $f_{16}$	  \\ \hline	
   $x_9$ & $x_{9}$ & $x_{10}$ & $-x_{4}-x_{11}$ & $-x_{4}-x_{11}$ & $x_{11}$ & $2x_{11}+x_{16}$ & $x_{15}$ \\ \hline\hline

   $f_{17}$ & $f_{18}$ & $f_{19}$ & $f_{20}$ & $f_{21}$ & $f_{22}$ &\cellcolor{Mygrey} &\cellcolor{Mygrey}\\ \hline
   $x_{15}$ & $-4x_{1}-4x_{11}-x_{17}$ & $x_{18}$ & $-x_1+x_{19}$ & $x_{20}$ & $x_{12}$ &\cellcolor{Mygrey} &\cellcolor{Mygrey}\\ \hline 
 \end{tabular}}
  }
  \label{tab: pairing3}
  \caption[Formulas of the update functions.]{\footnotesize The update functions of the network given by Fig. 2B in \cite{Zhang}. The functions are $f_1 = x_1$, $f_2=x_1$, $f_3 =x_2$, ..., $f_8 = 2x_3+x_5+3x_{14}$, etc.}
    \end{center}
\end{table} 
Let $A$ be the matrix of the above linear system with respect to the standard basis $(e_1,e_2,\ldots,e_{22})$. According to Theorem 2.1 of \cite{XZ}, the upper bound of the exponent for $A^m$ to stabilize (i.e. $A^{m+1}(\mathbb{Z}_9^{22}) = A^{m}(\mathbb{Z}_9^{22})$) is $22\ln(9) < 48$. So we take $N=48$ and let $B=A^{N}$. To find the cycle structure of $A$, we need to find the cycle structure of the induced linear map (the other part is the kernel of $B$):
\be
A:B(\mathbb{Z}_9^{22})\rightarrow B(\mathbb{Z}_9^{22}).
\ee 
\par
We make the following observation: if $U$ and $V$ are invertible matrices over $\mathbb{Z}_9$ such that $UBV$ is the Smith normal form for $B$, we can obtain the cycle lengths of $A$ from the cycle lengths of $UAU^{-1}$ as follows. Note that $UABV = UAU^{-1}(UBV)$ and $UBV$ is a diagonal matrix. Since $UBV$ gives the structure of the finite abelian group under consideration, we can read the action of $UAU^{-1}$ from the corresponding upper left submatrix of $UABV$ (i.e. we can just compute $UABV$). 
\par 
Our computation showed that the abelian group generated by the columns of $UBV$ is isomorphic to $\mathbb{Z}_9^4$ and the action induced by $A$ is given by the following matrix:
\be
S=\begin{pmatrix}
  1 & 0 & 1 & 0 \\
  0 & 1 & 2 & 0 \\
  0 & 0 & 0 & 1 \\
  0 & 0 & 1 & 0
\end{pmatrix}.
\ee
The elementary divisors of $S$ over the field $\mathbb{Z}_3$ are $x+2$, $x+1$, and $x^2+x+1$, and the corresponding orders are $1,2$, and $3$. Thus, according to Theorem \ref{T1}, the possible cycles lengths are the factors of $3\cdot 6=18$. By considering $A^kB = B$, we find that the order of $S$ is $18$. This can also be obtained by computing the order of $S$ modulo $9$ directly. Thus, the cycle lengths of the linear model are $1, 2, 3, 6, 9$, and $18$. All the computations can be done using the LinearAlgebra package of MAPLE in less than one second. We remark that the computation can also be done in about the same amount of time for the full $29$-node network.
\end{example}
\subsection*{{\bf Acknowledgment}}
Part of the work in this paper was done during the visit of YJW to the Department of Mathematical Sciences at the University of Wisconsin-Milwaukee in the summer of 2015. YJW wishes to thank the University of Wisconsin-Milwaukee and its faculty for the hospitality she received during her visit. 
 
\end{document}